\documentclass[10pt,reqno]{amsart}
\usepackage{amsmath}
\usepackage{amssymb}
\usepackage{amsthm}

\newlength{\defbaselineskip}
\newcommand{\setlinespacing}[1]%
           {\setlength{\baselineskip}{#1 \defbaselineskip}}

\numberwithin{equation}{section}

\newtheorem{thm}{Theorem}[section]
\newtheorem{prop}[thm]{Proposition}
\newtheorem{lem}[thm]{Lemma}

\theoremstyle{definition}

\theoremstyle{remark}
\newtheorem{rem}[thm]{Remark}
\numberwithin{equation}{section}

\begin{document}

\title[Absolute continuity of periodic Schr\"odinger operators]
{On absolute continuity of the spectrum of periodic Schr\"odinger operators}

\author{Ihyeok Seo}

\subjclass[2010]{Primary: 47A10; Secondary: 35J10}
\keywords{Spectrum, Schr\"odinger operators}

\address{Department of Mathematics, Sungkyunkwan University, Suwon 440-746, Republic of Korea}
\email{ihseo@skku.edu}

\maketitle

\begin{abstract}
In this paper we find a new condition on a real periodic potential for which
the self-adjoint Schr\"odinger operator may be defined by a quadratic form
and the spectrum of the operator is purely absolutely continuous.
This is based on resolvent estimates and spectral projection estimates in weighted $L^2$ spaces
on the torus, and an oscillatory integral theorem.
\end{abstract}


\section{Introduction}

The behavior of a non-relativistic quantum particle
is described by the wave function $\Psi(t,x)$
which is governed by the Schr\"odinger equation
$$i\partial_t\Psi(t,x)=H\Psi(t,x),\quad(t,x)\in\mathbb{R}\times\mathbb{R}^3,$$
where $H=-\Delta+V(x)$ is the Schr\"odinger operator and $V$ is the potential which is a real function.
In view of spectral theory, the solution can be given by
$\Psi(t,x)=e^{itH}\Psi(0,x)$ if $H$ is self-adjoint.
In this regard, spectral properties of
self-adjoint Schr\"odinger operators have been extensively studied since the beginning of quantum mechanics.

In this paper we are mainly concerned with the problem of finding conditions on a real periodic potential $V$
for which the spectrum of $H$ is purely absolutely continuous.
More generally, we will consider the following differential operator:
\begin{equation}\label{do}
DAD^T+V(x),\quad x\in\mathbb{R}^3,
\end{equation}
where $D=-i\nabla$ and $A=(a_{jk})$ is a symmetric, positive-definite $3\times3$ matrix with real constant entries.
Here, we are using $DAD^T$ to denote $\sum_{j,k=1}^3D_ja_{jk}D_k$,
and $V$ is a real periodic function which means that $V(x+e_j)=V(x)$
for some basis $\{e_j\}_{j=1}^3$ of $\mathbb{R}^3$.
Note that $DAD^T=-\Delta$ particularly when $A=I$ is the identity matrix,
and we may choose $e_1=2\pi(1,0,0)$, $e_2=2\pi(0,1,0)$ and $e_3=2\pi(0,0,1)$ by a change of variables.
Namely, $V$ is assumed to be periodic with respect to the lattice $(2\pi\mathbb{Z})^3$.

Let $\Omega=[0,2\pi)^3$ be a cell of the lattice, and for $N\geq0$
let $V_N(x)=V(x)$ if $|V(x)|>N$ and $V_N(x)=0$ if $|V(x)|\leq N$.
Let $\mathcal{F}$ be a function class equipped with the norm
$$\|f\|_{\mathcal{F}}:=\sup_{z\in\mathbb{R}^3,r>0}
\bigg(\int_{Q(z,r)}|f(x)|dx\bigg)^{-1}\int_{Q(z,r)}\int_{Q(z,r)}\frac{|f(x)f(y)|}{|x-y|}dxdy,$$
where $Q(z,r)$ denotes dyadic cubes in $\mathbb{R}^3$ centered at $z$ with side length $r$.
Now we will consider potentials $V\in L_{\textrm{loc}}^1{(\mathbb{R}^3)}$ such that
for a sufficiently small $\varepsilon>0$
\begin{equation}\label{ks}
\lim_{N\rightarrow\infty}\|V_N\chi_\Omega\|_{\mathcal{F}}<\varepsilon,
\end{equation}
where $\chi_\Omega$ is the characteristic function of the cell $\Omega$.
Note that the condition \eqref{ks} is equivalent to
$$\lim_{N\rightarrow\infty}\sup_{z\in\Omega,0<r<4\pi}
\bigg(\int_{Q(z,r)}|V_N(x)|dx\bigg)^{-1}\int_{Q(z,r)}\int_{Q(z,r)}\frac{|V_N(x)V_N(y)|}{|x-y|}dxdy<\varepsilon,$$
because $V_N$ is also periodic with respect to $(2\pi\mathbb{Z})^3$,
and $x,y$ is taking only in $\Omega$.

Let us now consider the following quadratic form to define the self-adjoint operator $DAD^T+V$:
\begin{equation}\label{qf}
q[f,g]=\int_{\mathbb{R}^3}\langle(Df)A,Dg\rangle+\langle Vf,g\rangle dx,\quad f,g\in C_0^1(\mathbb{R}^3),
\end{equation}
where $\langle\,\,,\,\rangle$ denotes the usual inner product in $\mathbb{C}$ or in $\mathbb{C}^3$.
Then we have

\begin{thm}\label{thm0}
Let $V\in L_{\textrm{loc}}^1{(\mathbb{R}^3)}$ be a real periodic function
with respect to the lattice $(2\pi\mathbb{Z})^3$.
If\, $V$ satisfies \eqref{ks},
then there exists a unique self-adjoint operator denoted by $DAD^T+V$ such that
$$q[f,g]=\int_{\mathbb{R}^3}\langle(DAD^T+V)f,g\rangle dx$$
for $f\in\textrm{Dom}\,[DAD^T+V]$ and $g\in H^1(\mathbb{R}^3)$.
Here, the domain of $DAD^T+V$ is
$$\textrm{Dom}\,[DAD^T+V]=\{f\in H^1(\mathbb{R}^3):(DAD^T+V)f\in L^2(\mathbb{R}^3)\}.$$
\end{thm}

Let $\mathcal{M}^p(\mathbb{R}^n)$ denote the Morrey class of functions $f$ on $\mathbb{R}^n$
which is defined by the norm
\begin{equation*}
\|f\|_{\mathcal{M}^{p}(\mathbb{R}^n)}:=\sup_{x\in\mathbb{R}^n,r>0}|Q|^{2/n}
\bigg(\frac1{|Q|}\int_{Q(x,r)}|f(y)|^pdy\bigg)^{1/p}<\infty,\quad1\leq p\leq n/2.
\end{equation*}
Note that $\mathcal{M}^{n/2}(\mathbb{R}^n)=L^{n/2}(\mathbb{R}^n)$ and $1/|x|^2\in L^{n/2,\infty}(\mathbb{R}^n)\subset\mathcal{M}^{p}(\mathbb{R}^n)$ if $p<n/2$.
In \cite{Sh3}, the self-adjointness of \eqref{do} was shown for $V\in\mathcal{M}^p(\mathbb{R}^n)$ with $p>(n-1)/2$, $n\geq3$, under the smallness assumption
\begin{equation}\label{assum2}
\limsup_{r\rightarrow0}\|V\|_{\mathcal{M}^{p}(\mathbb{R}^n)}<\varepsilon
\end{equation}
which implies (see Lemma 2.7 in \cite{Sh3})
\begin{equation}\label{fp}
\lim_{N\rightarrow\infty}\|V_N\chi_\Omega\|_{\mathcal{M}^{p}(\mathbb{R}^n)}<\varepsilon.
\end{equation}

\begin{rem}\label{rem}
Here we shall explain that \eqref{fp} with $n=3$ implies \eqref{ks}.
Thus Theorem 1.1 improves the one of \cite{Sh3} in three dimensions.
In fact, our motivation for the function class $\mathcal{F}$ stemmed from the characterization
of the weighted $L^2$ inequalities
\begin{equation}\label{wi}
\|I_1f\|_{L^2(w)}\leq C_w\|f\|_{L^2},
\end{equation}
where $I_\alpha$ denotes the fractional integral operator of order $0<\alpha<n$:
\begin{equation}\label{fi}
I_\alpha f(x)=\int_{\mathbb{R}^n}\frac{f(y)}{|x-y|^{n-\alpha}}dy.
\end{equation}
The class $\mathcal{F}$ and the characterization of \eqref{wi} has been recently used
in various problems concerning Schr\"odinger operators and equations (\cite{BBRV,S,S2,S3}).
As is well known from \cite{F}, \eqref{wi} holds for $w\in\mathcal{M}^p(\mathbb{R}^n)$
with $C_w\sim\|w\|_{\mathcal{M}^{p}(\mathbb{R}^n)}^{1/2}$ if $p>1$.
But here, we note that the least constant $C_w$ which allows \eqref{wi} with $n=3$
may be taken to be a constant multiple of $\|w\|_{\mathcal{F}}^{1/2}$
(see Lemma \ref{lem} for more details).
It is therefore clear that \eqref{fp} with $n=3$ implies \eqref{ks}.
\end{rem}

Now we turn to the absolute continuity.
We first need to set up more notation.
A weight $w:\mathbb{R}^{3}\rightarrow[0,\infty]$ is said to be of Muckenhoupt $A_2(\mathbb{R}^{3})$ class
({\it cf. \cite{G}})
if there is a constant $C_{A_2}$ such that
\begin{equation*}
\sup_{Q\text{ cubes in }\mathbb{R}^{3}}
\bigg(\frac1{|Q|}\int_Qw(x)dx\bigg)\bigg(\frac1{|Q|}\int_Qw(x)^{-1}dx\bigg)<C_{A_2}.
\end{equation*}
Note that $w\in A_2\,\,\Leftrightarrow\,\, w^{-1}\in A_2$.
Given $v\in\mathbb{R}^{3}$, one can write for $x\in\mathbb{R}^{3}$, $x=\lambda v+\widetilde{x}$,
where $\lambda\in\mathbb{R}$ and $\widetilde{x}$ is in some hyperplane $\mathcal{P}$
whose normal vector is $v$.
We shall denote by $w\in A_p(v)$ to mean that $w$ is in the $A_2$ class
in one-dimensional direction of the vector $v$
if the function $w_{\widetilde{x}}(\lambda):=w(x)$ is in $A_2(\mathbb{R})$ with $C_{A_2}$
uniformly in almost every $\widetilde{x}\in\mathcal{P}$.
By translation and rotation, this notion can be reduced to the case
where $v=(0,0,1)\in\mathbb{R}^{3}$ and $\mathcal{P}=\mathbb{R}^{2}$.
In this case, $w\in A_2(v)$ means that $w(x_1,x_2,\cdot)\in A_2(\mathbb{R})$
with respect to the variable $x_{3}$ uniformly for $\widetilde{x}=(x_1,x_2)\in\mathbb{R}^{2}$.
Also, $w\in A_2(v)$ is trivially satisfied
if $w$ is in a more restrictive $A_2(\mathbb{R}^{3})$ class defined over rectangles instead of cubes
(see Lemma 2.2 in \cite{Ku}).

This one-dimensional $A_2$ condition was already appeared in the study of
unique continuation problems ({\it cf. \cite{CR,S2}}),
and will be needed here for our resolvent estimates in Proposition \ref{prop}
which is a key ingredient in the proof of the following theorem which gives a new condition
concerning the absolute continuity:

\begin{thm}\label{thm}
Let $V\in L_{\textrm{loc}}^1{(\mathbb{R}^3)}$ be a real periodic function
with respect to the lattice $(2\pi\mathbb{Z})^3$.
If\, $V$ satisfies the conditions \eqref{ks} and $|V|\in A_2(v)$ for some $v\in\mathbb{R}^3$,
then the spectrum of $DAD^T+V(x)$ is purely absolutely continuous.
\end{thm}

Making use of resolvent estimates for a family of operators $(D+k)^2+V$, $k\in\mathbb{C}^n$,
Thomas \cite{T} showed that the spectrum of the Schr\"odinger operator is purely absolutely continuous
if $V\in L_{\textrm{loc}}^2(\mathbb{R}^3)$.
Based on this approach, the absolute continuity of periodic operators has been extensively studied
by many authors (\cite{RS,D,HH,HH2,BS,BS2,BS3,M,So,Sh,Sh2,Sh3,SZ}).
Among others,
Shen \cite{Sh} established the absolute continuity of \eqref{do} for $V\in L_{\textrm{loc}}^{n/2}(\mathbb{R}^n)$, $n\geq3$,
which is best possible in the context of $L^p$ potentials.
This result was later extended by himself \cite{Sh3} to the Morrey class
$\mathcal{M}^p(\mathbb{R}^n)$ with $p>(n-1)/2$, $n\geq3$,
under the smallness assumption \eqref{assum2}.
See also \cite{Sh2} for the Kato class in three dimensions.

Finally, we would like to emphasize that the class $\mathcal{F}$
contains the global Kato ($\mathcal{K}$) and Rollnik ($\mathcal{R}$) classes which are defined by
\begin{equation*}
\|f\|_{\mathcal{K}}:=
\sup_{x\in\mathbb{R}^3}\int_{\mathbb{R}^3}\frac{|f(y)|}{|x-y|}dy<\infty
\end{equation*}
and
\begin{equation*}
\|f\|_{\mathcal{R}}:=
\int_{\mathbb{R}^3}\int_{\mathbb{R}^3}\frac{|f(x)f(y)|}{|x-y|^2}dxdy<\infty,
\end{equation*}
respectively.
These are fundamental ones in spectral and scattering theory ({\it cf. \cite{K,Si}}),
and their usefulness was recently revealed in the work of Rodnianski and Schlag \cite{RoS}
concerning dispersive properties of Schr\"odinger equations.

\medskip

The rest of the paper is organized as follows.
In Section \ref{sec2}, we prove Theorem \ref{thm0} which says that
the self-adjoint operator $DAD^T+V$ can be defined by a quadratic form under the condition \eqref{ks}.
Based on the Thomas approach, we make use of the weighted $L^2$ resolvent estimates in Proposition \ref{prop} to prove Theorem \ref{thm} in Section \ref{sec3}.
Sections \ref{sec4} and \ref{sec5} are devoted to proving Proposition \ref{prop}.
The key ingredient in the proof is the weighted $L^2$ spectral projection estimates of Proposition \ref{sam}
in Section \ref{sec4}.
These estimates will be shown in Section \ref{sec5} by using Lemma \ref{lem5} which is a weighted version of an oscillatory integral theorem of Stein.

\medskip

From now on, we will use the letter $C$ for positive constants that may be different at each occurrence.


\section{Self-adjointness}\label{sec2}

In this section we prove Theorem \ref{thm0}.
Namely, we will show that the operator $DAD^T+V$ is self-adjoint
under the condition \eqref{ks} on a real periodic $V$.

Let $\psi\in C_0^1(\mathbb{R}^3)$.
From the definition of $V_N$, we see that
\begin{equation}\label{1}
\int_{\mathbb{R}^3}|\psi|^2|V|dx\leq\int_{\mathbb{R}^3}|\psi|^2|V_N|dx+N\int_{\mathbb{R}^3}|\psi|^2dx.
\end{equation}
Now we claim that
\begin{equation}\label{claim}
\int_{\mathbb{R}^3}|\psi|^2|V_N|dx\leq
C\|V_N\chi_\Omega\|_{\mathcal{F}}\int_{\mathbb{R}^3}|\nabla\psi|^2dx
+C\int_\Omega|V_N|dx\int_{\mathbb{R}^3}|\psi|^2dx.
\end{equation}
Assuming this, we get from \eqref{1} that
$$\int_{\mathbb{R}^3}|\psi|^2|V|dx\leq
C\|V_N\chi_\Omega\|_{\mathcal{F}}\int_{\mathbb{R}^3}|\nabla\psi|^2dx
+\bigg(C\int_\Omega|V_N|dx+N\bigg)\int_{\mathbb{R}^3}|\psi|^2dx.$$
Since $\langle A\nabla\psi,\nabla\psi\rangle\geq C|\nabla\psi|^2$ and $V\in L_{\textrm{loc}}^1{(\mathbb{R}^3)}$,
by the condition \eqref{ks}, we conclude that
\begin{equation}\label{claim2}
\int_{\mathbb{R}^3}|\psi|^2|V|dx\leq
C\varepsilon\int_{\mathbb{R}^3}\langle A\nabla\psi,\nabla\psi\rangle dx+C_{N}\int_{\mathbb{R}^3}|\psi|^2dx
\end{equation}
if $N$ is sufficiently large.
Hence, if $\varepsilon$ is small enough so that $C\varepsilon<1/2$,
then \eqref{claim2} clearly implies that the symmetric quadratic form $q$ given in \eqref{qf}
is semi-bounded from below and closable on $H^1(\mathbb{R}^3)$.
Thus, it defines a unique self-adjoint operator, which we denote by $DAD^T+V$,
such that
$$q[f,g]=\int_{\mathbb{R}^3}\langle(DAD^T+V)f,g\rangle dx$$
for $f\in\textrm{Dom}\,[DAD^T+V]$ and $g\in H^1(\mathbb{R}^3)$.
Also,
$$\textrm{Dom}\,[DAD^T+V]=\{f\in H^1(\mathbb{R}^3):(DAD^T+V)f\in L^2(\mathbb{R}^3)\}.$$

Now it remains to show the claim \eqref{claim}.
For this, we first note that
$$\bigg|\psi(y)-\frac1{|\Omega|}\int_\Omega\psi(x)dx\bigg|
\leq CI_1(|\nabla\psi|\chi_\Omega)(y),$$
where $I_1$ is the fractional integral operator of order $1$ given in \eqref{fi}
(see Lemma 7.16 in \cite{GT}).
Then, by using this, $(a^{1/2}+b^{1/2})^2\leq2(a+b)$, and H\"older's inequality,
it is not difficult to see that
\begin{equation}\label{2}
\int_\Omega|\psi|^2|V_N|dx\leq
C\int_{\mathbb{R}^3}|I_1(|\nabla\psi|\chi_\Omega)|^2|V_N|\chi_\Omega dx
+C\int_\Omega|V_N|dx\int_\Omega|\psi|^2dx.
\end{equation}
Now we will use the following lemma,
which characterizes weighted $L^2$ inequalities for the fractional integral operator $I_1$,
due to Kerman and Sawyer \cite{KS} (see Theorem 2.3 there and also Lemma 2.1 in \cite{BBRV}):

\begin{lem}\label{lem}
Let $w$ be a nonnegative measurable function on $\mathbb{R}^3$.
Then there exists a constant $C_w$ depending on $w$ such that
\begin{equation}\label{ee}
\|I_1f\|_{L^2(w)}\leq C_w\|f\|_{L^2}
\end{equation}
for all measurable functions $f$ on $\mathbb{R}^3$
if and only if
\begin{equation*}
\|w\|_{\mathcal{F}}:=\sup_{Q}\bigg(\int_Qw(x)dx\bigg)^{-1}\int_Q\int_Q\frac{w(x)w(y)}{|x-y|}dxdy<\infty.
\end{equation*}
Here the sup is taken over all dyadic cubes $Q$ in $\mathbb{R}^3$,
and the constant $C_w$ may be taken to be a constant multiple of $\|w\|_{\mathcal{F}}^{1/2}$.
\end{lem}
Applying this lemma to the first term in the right-hand side of \eqref{2}, we see that
\begin{equation}\label{22}
\int_{\mathbb{R}^3}|I_1(|\nabla\psi|\chi_\Omega)|^2|V_N|\chi_\Omega dx
\leq C\|V_N\chi_\Omega\|_{\mathcal{F}}\int_\Omega|\nabla\psi|^2dx.
\end{equation}
Combining \eqref{2} and \eqref{22}, we now get
\begin{equation}\label{90}
\int_\Omega|\psi|^2|V_N|dx\leq
C\|V_N\chi_\Omega\|_{\mathcal{F}}\int_\Omega|\nabla\psi|^2dx
+C\int_\Omega|V_N|dx\int_\Omega|\psi|^2dx
\end{equation}
which readily implies the claim \eqref{claim} because $V_N$ is periodic
with respect to the cell $\Omega$.

\section{Absolute continuity}\label{sec3}

In this section we prove Theorem \ref{thm}.
For $k\in\mathbb{C}^3$,
we need to define the operator $(D+k)A(D+k)^T+V$ on $L^2(\mathbb{T}^3)$,
where $\mathbb{T}^3=\mathbb{R}^3/(2\pi\mathbb{Z})^3\approx[0,2\pi)^3=\Omega$.
For this, we first let
$$H^1(\mathbb{T}^3)=\bigg\{\psi\in L^2(\Omega):\psi(x)=\sum_{\textbf{n}\in\mathbb{Z}^3}a_\textbf{n}e^{i\textbf{n}\cdot x}\text{ and }
\sum_{\textbf{n}\in\mathbb{Z}^3}|\textbf{n}|^2|a_\textbf{n}|^2<\infty\bigg\}.$$
Let us then consider the following quadratic form depending on $k$:
\begin{equation}\label{form}
q(k)[\phi,\psi]=\int_{\Omega}\langle[(D+k)\phi]A,(D+\overline{k})\psi\rangle+\langle V\phi,\psi\rangle dx,
\end{equation}
where $\phi,\psi\in H^1(\mathbb{T}^3)$ and $\overline{k}$ denotes the conjugate of $k$.
Now we observe that
\begin{equation}\label{99}
\int_{\Omega}|\psi|^2|V|dx\leq
C\varepsilon\int_{\Omega}\langle A\nabla\psi,\nabla\psi\rangle dx+C_{N}\int_{\Omega}|\psi|^2dx
\end{equation}
if $N$ is large.
Indeed, from the definition of $V_N$,
\begin{equation*}
\int_\Omega|\psi|^2|V|dx\leq\int_\Omega|\psi|^2|V_N|dx+N\int_\Omega|\psi|^2dx.
\end{equation*}
So, \eqref{99} follows from combining this, \eqref{90} and \eqref{ks}.
If we choose $\varepsilon$ in \eqref{99} small enough so that $C\varepsilon<1/2$,
then this implies that the quadratic form $q(k)$ is strictly m-sectorial.
Thus, there exists a unique closed operator, which we denote by $(D+k)A(D+k)^T+V$,
such that
$$q(k)[\phi,\psi]=\int_{\Omega}\langle[(D+k)A(D+k)^T+V]\phi,\psi\rangle dx$$
for $\phi\in\text{Dom}[(D+k)A(D+k)^T+V]$ and $\psi\in H^1(\mathbb{T}^3)$.
(See \cite{K2} for details.)
Also,
\begin{align*}
\text{Dom}[(D+k)A(D+k)^T+V]&=\{\phi\in H^1(\mathbb{T}^3):[(D+k)A(D+k)^T+V]\phi\in L^2(\mathbb{T}^3)\}\\
&=\{\phi\in H^1(\mathbb{T}^3):(DAD^T+V)\phi\in L^2(\mathbb{T}^3)\}
\end{align*}
is independent of $k$.

Next we choose $\textbf{a}=(a_1,a_2,a_3)\in\mathbb{R}^3$ such that
\begin{equation}\label{7}
|\textbf{a}|=1\quad\text{and}\quad \textbf{a}A=(s_0,0,0),\, s_0>0,
\end{equation}
and let
\begin{equation}\label{8}
L=\{\textbf{b}=(b_1,b_2,b_3)\in\mathbb{R}^3:|\textbf{b}|<\sqrt{3}\text{ and } \langle\textbf{b},\textbf{a}\rangle=0\}.
\end{equation}
For a fixed $\textbf{b}\in L$, let us now consider a family of operators
$$H_V(\lambda)=(D+\lambda \textbf{a}+\textbf{b})A(D+\lambda \textbf{a}+\textbf{b})^T+V,\quad \lambda\in\mathbb{C},$$
defined by the quadratic form \eqref{form}.
Then the following lemma is standard (see Propositions 4.5 and 3.9 in \cite{Sh2} and \cite{Sh3}, respectively).

\begin{lem}\label{keylem}
If, for every $b\in L$, the family of operators $\{H_V(\lambda):\lambda\in\mathbb{C}\}$
has no common eigenvalue, then the spectrum of the self-adjoint operator $DAD^T+V$
is purely absolutely continuous.
\end{lem}

To prove Theorem \ref{thm}, by this lemma
we only need to show that $\{H_V(\lambda):\lambda\in\mathbb{C}\}$ has no common eigenvalue.
For this, let us first consider $\lambda=\delta_0+i\rho$, where
$$\delta_0=\frac1{a_1}(\frac12-b_1).$$
Here $a_1$ and $b_1$ are given in \eqref{7} and \eqref{8}, respectively.
Since $\langle \textbf{a}A,\textbf{a}\rangle=a_1s_0$ and $|\textbf{a}|=1$,
$a_1\neq0$.
From now on, we will show that $\{H_V(\delta_0+i\rho):\rho\in\mathbb{R}\}$ has no common eigenvalue.
This will be based on the following weighted $L^2$ resolvent estimates
which will be obtained in the next section.

\begin{prop}\label{prop}
Let $w\in L^1(\mathbb{T}^3)$ and $w\in A_2(v)$ for some $v\in\mathbb{R}^3$.
Assume that $w\geq c_w$ for some constant $c_w>0$ and
\begin{equation*}
\|w\|_{\mathcal{F}(\mathbb{T}^3)}:=\sup_{z\in\Omega,0<r<4\pi}
\bigg(\int_{Q(z,r)}w(x)dx\bigg)^{-1}\int_{Q(z,r)}\int_{Q(z,r)}\frac{w(x)w(y)}{|x-y|}dxdy<\infty.
\end{equation*}
Then, if $\psi\in H^1(\mathbb{T}^3)$ and $H_0(\delta_0+i\rho)\psi\in L^2(\mathbb{T}^3,w^{-1}dx)$,
we have for $|\rho|\geq1$
\begin{equation}\label{resol}
\|\psi\|_{L^2(\mathbb{T}^3,wdx)}\leq C\|w\|_{\mathcal{F}(\mathbb{T}^3)}\|H_0(\delta_0+i\rho)\psi\|_{L^2(\mathbb{T}^3,w^{-1}dx)},
\end{equation}
where $C$ is a constant independent of $\rho$ and $c_w$.
\end{prop}

\begin{rem}
The estimate \eqref{resol} is a uniform Sobolev inequality on the torus $\mathbb{T}^3$
for the second-order elliptic operator $H_0(\delta_0+i\rho)$.
Similar inequalities were obtained in the setting of $\mathbb{R}^n$ by
many authors (\cite{KRS,CS,CR,S2}) to study unique continuation properties of differential operators.
Also, \eqref{resol} was shown in \cite{Sh3} for $w$ in the Morrey class.
\end{rem}

Now we suppose that $E$ is an eigenvalue for $H_V(\lambda)$ for all $\lambda\in\mathbb{C}$.
Then there exists $\psi_\rho\in\text{Dom}[H_V(\lambda)]$ particularly for $\lambda=\delta_0+i\rho$
such that $\|\psi_\rho\|_{L^2(\mathbb{T}^3)}=1$
and $H_V(\lambda)\psi_\rho=E\psi_\rho$.
By the condition \eqref{ks}, we can choose $N$ so large that
\begin{equation*}
\sup_{z\in\Omega,0<r<4\pi}
\bigg(\int_{Q(z,r)}|V_N(x)|dx\bigg)^{-1}\int_{Q(z,r)}\int_{Q(z,r)}\frac{|V_N(x)V_N(y)|}{|x-y|}dxdy<\varepsilon.
\end{equation*}
Let us now consider $w(x)=|V_N(x)|+f_\delta(x)$,
where $f_\delta$ is a periodic function with respect to $(2\pi\mathbb{Z})^3$,
which is given by $f_\delta(x)=\delta/|x|^2$ with $\delta>0$ for $x\in\Omega$.
Then, $w$ is periodic with respect to $(2\pi\mathbb{Z})^3$ and $w\geq\widetilde{c}>0$ for some constant $\widetilde{c}$.
Recall from Remark \ref{rem} that
$\|f_\delta\chi_{\Omega}\|_{\mathcal{F}}\leq C\|f_\delta\chi_{\Omega}\|_{\mathcal{M}^p(\mathbb{R}^3)}\leq C\delta$
and that
the norm $\|\cdot\|_{\mathcal{F}}$ is the least bound for
$$\int|I_1f(x)|^2w(x)dx\leq C\|w\|_{\mathcal{F}}\int|f(x)|^2dx.$$
It is now clear that
\begin{equation}\label{sump}
\|w_1+w_2\|_{\mathcal{F}}\leq C(\|w_1\|_{\mathcal{F}}+\|w_2\|_{\mathcal{F}}).
\end{equation}
Since $\|w\|_{\mathcal{F}(\mathbb{T}^3)}=C\|w\chi_{\Omega}\|_{\mathcal{F}}$
(see the paragraph below \eqref{ks}),
from \eqref{sump} we can take $\delta$ small enough so that
\begin{equation}\label{ks22}
\sup_{z\in\Omega,0<r<4\pi}
\bigg(\int_{Q(z,r)}w(x)dx\bigg)^{-1}\int_{Q(z,r)}\int_{Q(z,r)}\frac{w(x)w(y)}{|x-y|}dxdy<\varepsilon.
\end{equation}
Using Lemma \ref{lem} together with \eqref{ks22}, we now get
\begin{align}\label{47}
\nonumber\int_\Omega|V\psi_\rho|^2w^{-1}dx&\leq N^2\int_{\Omega\cap\{V(x)\leq N\}}|\psi_\rho|^2w^{-1}dx
+\int_{\Omega\cap\{V(x)>N\}}|\psi_\rho|^2|V_N|^2w^{-1}dx\\
&\leq N^2\int_{\Omega}|\psi_\rho|^2w^{-1}dx
+\int_{\Omega}|\psi_\rho|^2wdx\\
\nonumber&\leq N^2\widetilde{c}^{-1}\int_{\Omega}|\psi_\rho|^2dx
+C\varepsilon\int_{\Omega}|\nabla\psi_\rho|^2dx.
\end{align}
Hence, $V\psi_\rho\in L^2(\mathbb{T}^3,w^{-1}dx)$ because $\psi_\rho\in H^1(\mathbb{T}^3)$.
Also, $\psi_\rho\in L^2(\mathbb{T}^3,w^{-1}dx)$ since $w\geq\widetilde{c}>0$.
Since $H_V(\lambda)\psi_\rho=E\psi_\rho$, we now conclude that
$$H_0(\delta_0+i\rho)\psi_\rho=E\psi_\rho-V\psi_\rho
\in L^2(\mathbb{T}^3,w^{-1}dx),$$
and from \eqref{47}
\begin{equation}\label{45}
\|H_0(\delta_0+i\rho)\psi_\rho\|_{L^2(\mathbb{T}^3,w^{-1}dx)}
\leq(N+|E|)\|\psi_\rho\|_{L^2(\mathbb{T}^3,w^{-1}dx)}+\|\psi_\rho\|_{L^2(\mathbb{T}^3,wdx)}.
\end{equation}
On the other hand, by Proposition \ref{prop} with \eqref{ks22}, we see that for $\rho\geq1$
\begin{equation*}
\|\psi_\rho\|_{L^2(\mathbb{T}^3,wdx)}
\leq C\varepsilon
\|H_0(\delta_0+i\rho)\psi_\rho\|_{L^2(\mathbb{T}^3,w^{-1}dx)}.
\end{equation*}
So if $\varepsilon$ is chosen so small that $C\varepsilon\leq1/2$, from \eqref{45} we have
\begin{equation}\label{comm}
\|H_0(\delta_0+i\rho)\psi_\rho\|_{L^2(\mathbb{T}^3,w^{-1}dx)}
\leq2(N+|E|)\|\psi_\rho\|_{L^2(\mathbb{T}^3,w^{-1}dx)}.
\end{equation}
Now we recall from Theorem 3.13 in \cite{Sh3} that
there exists $c_\rho>0$ such that $c_\rho\rightarrow\infty$ as $|\rho|\rightarrow\infty$ and
$$c_\rho c_w^{1/2}\|\psi\|_{L^2(\mathbb{T}^3,w^{-1}dx)}\leq
\|w\|_{L^1(\mathbb{T}^3)}^{1/2}\|H_0(\delta_0+i\rho)\psi\|_{L^2(\mathbb{T}^3,w^{-1}dx)}$$
if $\psi\in H^1(\mathbb{T}^3)$, $H_0(\delta_0+i\rho)\psi\in L^2(\mathbb{T}^3,w^{-1}dx)$,
and $w\in L^1(\mathbb{T}^3)$ with $w\geq c_w>0$.
By combining this and \eqref{comm}, it follows that
\begin{equation*}
c_\rho c_w^{1/2}\leq\|w\|_{L^1(\mathbb{T}^3)}^{1/2}2(N+|E|)<\infty.
\end{equation*}
This leads to a contradiction
since $c_\rho\rightarrow\infty$ as $\rho\rightarrow\infty$.
Thus $\{H_V(z):z\in\mathbb{C}\}$ has no common eigenvalue,
and so Theorem \ref{thm} is proved by Lemma \ref{keylem}.

\section{Weighted $L^2$ resolvent estimates}\label{sec4}

This section is devoted to proving Proposition \ref{prop}.
By an elementary rotation argument, we may first assume that
$w\in A_2(\mathbb{R})$ in the $x_1$ variable uniformly in other variables $x'=(x_2,x_3)\in\mathbb{T}^2$.
Then we need to show that for $\psi\in L^2(\mathbb{T}^3,w^{-1}dx)$,
\begin{equation}\label{resol2}
\|H_0(\delta_0+i\rho)^{-1}\psi\|_{L^2(\mathbb{T}^3,wdx)}\leq C\|w\|_{\mathcal{F}(\mathbb{T}^3)}\|\psi\|_{L^2(\mathbb{T}^3,w^{-1}dx)},
\end{equation}
where
$$H_0(\delta_0+i\rho)^{-1}\psi(x)=\sum_{\textbf{n}=(n_1,n_2,n_3)\in\mathbb{Z}^3}\frac{\widehat{\psi}(\textbf{n})e^{i\textbf{n}\cdot x}}{(\textbf{n}+k)A(\textbf{n}+k)^T}$$
and $k=(\delta_0+i\rho)\textbf{a}+\textbf{b}$ with
$\textbf{a},\textbf{b}$ given as in \eqref{7},\eqref{8}.
To show \eqref{resol2}, we first decompose $\psi$ as
$\psi=\sum_{j=-\infty}^\infty\psi_j$,
where
$$\psi_j=\sum_{n_1\in[2^{j-1},2^j-1]}\widehat{\psi}(\textbf{n})e^{i\textbf{n}\cdot x}
\quad\text{for}\quad j\geq1,$$
$$\psi_j=\sum_{n_1\in[-2^{-j}+1,-2^{-j-1}]}\widehat{\psi}(\textbf{n})e^{i\textbf{n}\cdot x}
\quad\text{for}\quad j\leq-1,$$
and
$$\psi_0=\sum_{n_1=0}\widehat{\psi}(\textbf{n})e^{i\textbf{n}\cdot x}.$$
Then by the Littlewood-Paley theory on $\mathbb{T}^1$ in weighted $L^2$ spaces
(see Chap. XV in \cite{Z} and also \cite{Ku}),
it is enough to show \eqref{resol2} for $\psi_j$ uniformly in $j$.
Indeed, if we have
\begin{equation}\label{resol3}
\|H_0(\delta_0+i\rho)^{-1}\psi_j\|_{L^2(\mathbb{T}^3,wdx)}\leq C\|w\|_{\mathcal{F}(\mathbb{T}^3)}\|\psi_j\|_{L^2(\mathbb{T}^3,w^{-1}dx)}
\end{equation}
uniformly in $j$,
then by the Littlewood-Paley theory and the condition $w\in A_2(\mathbb{R})$, we get
\begin{align*}
\|H_0(\delta_0+i\rho)^{-1}\psi\|_{L^2(\mathbb{T}^3,wdx)}
&\leq C\bigg\|\bigg(\sum_{j\in\mathbb{Z}}
|H_0(\delta_0+i\rho)^{-1}\psi_j|^2\bigg)^{1/2}\bigg\|_{L^2(\mathbb{T}^3,wdx)}\\
&\leq C\|w\|_{\mathcal{F}(\mathbb{T}^3)}
\bigg\|\bigg(\sum_{j\in\mathbb{Z}}|\psi_j|^2\bigg)^{1/2}\bigg\|_{L^2(\mathbb{T}^3,w^{-1}dx)}\\
&\leq C\|w\|_{\mathcal{F}(\mathbb{T}^3)}\|\psi\|_{L^2(\mathbb{T}^3,w^{-1}dx)}
\end{align*}
as desired.

From now on, we will show \eqref{resol3} only for the case $j\geq1$
because the other case $j\leq0$ can be shown in the same way.
First we recall from (6.5) in \cite{Sh3} (see also (3.5) in \cite{Sh}) that
$$(\textbf{n}+k)A(\textbf{n}+k)^T
=|(\textbf{n}+\textbf{b})B|^2+2\delta_0(n_1+b_1)s_0+(\delta_0^2-\rho^2)a_1s_0+2i\rho(n_1+\frac12)s_0$$
where $B$ is a $3\times3$ symmetric, positive definite matrix such that $A=B^2$ (i.e., $B=\sqrt{A}$).
In fact, this follows easily from \eqref{7} and \eqref{8}.
Fix $j\geq1$. In view of the fact that $n_1+1/2\sim2^j$, we let
$z_j=-\rho^2a_1s_0+2i\rho2^js_0$,
and we consider the following operator
$$((D+\textbf{b})A(D+\textbf{b})^T+z_j)^{-1}\psi=
\sum_{\textbf{n}\in\mathbb{Z}^3}\frac{\widehat{\psi}(\textbf{n})e^{i\textbf{n}\cdot x}}{|(\textbf{n}+\textbf{b})B|^2+z_j}.$$
Now we will show that
\begin{equation}\label{resol23}
\bigg\|\sum_{\textbf{n}\in\mathbb{Z}^3}\frac{\widehat{\psi_j}(\textbf{n})e^{i\textbf{n}\cdot x}}{|(\textbf{n}+\textbf{b})B|^2+z_j}\bigg\|_{L^2(\mathbb{T}^3,wdx)}\leq C\|w\|_{\mathcal{F}(\mathbb{T}^3)}\|\psi_j\|_{L^2(\mathbb{T}^3,w^{-1}dx)}
\end{equation}
and
\begin{align}\label{spli}
\nonumber\bigg\|H_0(\delta_0+\rho)^{-1}\psi_j-
\sum_{\textbf{n}\in\mathbb{Z}^3}&\frac{\widehat{\psi_j}(\textbf{n})e^{i\textbf{n}\cdot x}}
{|(\textbf{n}+\textbf{b})B|^2+z_j}\bigg\|_{L^2(\mathbb{T}^3,wdx)}\\
&\qquad\qquad\leq C\|w\|_{\mathcal{F}(\mathbb{T}^3)}\|\psi_j\|_{L^2(\mathbb{T}^3,w^{-1}dx)}.
\end{align}
Then the desired estimate \eqref{resol3} follows directly from combining \eqref{resol23} and \eqref{spli},
and so the proof of Proposition \ref{prop} is completed.

To show the first estimate \eqref{resol23}, we consider the family of operators
\begin{equation*}
S_\xi\psi(x)=
\sum_{\textbf{n}\in\mathbb{Z}^3}\frac{\widehat{\psi_j}(\textbf{n})
e^{i\textbf{n}\cdot x}}{[(\textbf{n}+\textbf{b})A(\textbf{n}+\textbf{b})^T+z_j]^\xi},\quad\xi\in\mathbb{C},
\end{equation*}
where
\begin{align*}
\text{Re}\sqrt{z_j}=|z_j|^{1/2}\cos\Big(\frac12\arg(z_j)\Big)
&\geq\frac12\frac{|\text{Im}z_j|}{|z_j|^{1/2}}\\
&\geq\frac{c|\rho|2^j}{|\rho|+\sqrt{|\rho|2^j}}\\
&\geq c\min(2^j,\sqrt{|\rho|2^j})\geq c_0>0.
\end{align*}
Then we have
$$S_\xi\psi(x)=\int_\Omega G_\xi(x-y)\psi(y)dy$$
and the integral kernel $G_\xi$ of $S_\xi$ satisfies
$$|G_\xi(x)|\leq Ce^{c|\text{Im}\,\xi|}\bigg(1+\sum_{|x+2\pi\textbf{n}|\leq C}\frac{1}{|x+2\pi\textbf{n}|}\bigg)$$
for $\text{Re}\,\xi=1$. See (6.10) in \cite{Sh}.
It follows now that
\begin{align*}
|S_\xi\psi(x)|&\leq Ce^{c|\text{Im}\,\xi|}
\bigg(\int_\Omega|\psi(y)|dy+\int_\Omega\sum_{|x-y+2\pi\textbf{n}|\leq C}\frac{\psi(y)}{|x-y+2\pi\textbf{n}|}dy\bigg)\\
&\leq Ce^{c|\text{Im}\,\xi|}
\bigg(\int_\Omega|\psi(y)|dy+I_2(|\psi|\chi_{\Omega'})(x)\bigg),
\end{align*}
where $\Omega'=\bigcup_{|\textbf{n}|\leq C}(\Omega+2\pi\textbf{n})$
and $I_2$ is the fractional integral operator of order $2$ given in \eqref{fi}.
Here, for the last inequality we used the fact that $\psi$ is periodic.
Hence, for $\text{Re}\,\xi=1$
\begin{align}\label{35}
\nonumber|w^{\xi/2}S_\xi&(w^{\xi/2}\psi)(x)|\\
\nonumber\leq C&e^{c|\text{Im}\,\xi|}
\bigg(w(x)^{1/2}\int_\Omega|\psi(y)|w(y)^{1/2}dy+w(x)^{1/2}I_2(|\psi|\chi_{\Omega'}w^{1/2})(x)\bigg)\\
:=C&e^{c|\text{Im}\,\xi|}(I+II).
\end{align}
Then, using H\"older's inequality, we can bound the first term in the right-hand side of \eqref{35} as
\begin{align*}
\|I\|_{L^2(\Omega,dx)}&\leq
\int_\Omega|\psi(y)|w(y)^{1/2}dy\bigg(\int_\Omega w(x)dx\bigg)^{1/2}\\
&\leq\|\psi\|_{L^2(\Omega)}\int_\Omega w(x)dx\\
&\leq C\|\chi_\Omega w\|_{\mathcal{F}}\|\psi\|_{L^2(\Omega)}\\
&\leq C\|w\|_{\mathcal{F}(\mathbb{T}^3)}\|\psi\|_{L^2(\Omega)}.
\end{align*}
Here, for the last inequality we used the fact that $w$ is periodic.
Since $1/|x-y|\geq c>0$ for $x,y\in\Omega$, we see here that
\begin{align*}
c\int_\Omega w(x)dx
&\leq\bigg(\int_{\Omega}w(x)dx\bigg)^{-1}\int_{\Omega}\int_{\Omega}\frac{w(x)w(y)}{|x-y|}dxdy\\
&\leq
\sup_{z\in\mathbb{R}^3,r>0}\bigg(\int_{Q(z,r)\cap\Omega}w(x)dx\bigg)^{-1}
\int_{Q(z,r)\cap\Omega}\int_{Q(z,r)\cap\Omega}\frac{w(x)w(y)}{|x-y|}dxdy\\
&=\|\chi_\Omega w\|_{\mathcal{F}}.
\end{align*}
On the other hand, for the second term we will use the following estimate
\begin{equation}\label{ee3}
\|I_2f\|_{L^2(w)}\leq C\|w\|_{\mathcal{F}}\|f\|_{L^2(w^{-1})},
\end{equation}
which follows by combining \eqref{ee} in Lemma \ref{lem} and its dual estimate
\begin{equation*}
\|I_1f\|_{L^2}\leq C_w\|f\|_{L^2(w^{-1})}
\end{equation*}
since $I_2=I_1I_1$.
Using \eqref{ee3}, we now see that
\begin{align*}\label{resol233}
\|II\|_{L^2(\Omega,dx)}
&\leq\|\chi_{\Omega'}w^{1/2}I_2(|\psi|\chi_{\Omega'}w^{1/2})\|_{L^2}\\
&\leq C\|\chi_{\Omega'}w\|_{\mathcal{F}}\|\psi\chi_{\Omega'}\|_{L^2}\\
&\leq C\|w\|_{\mathcal{F}(\mathbb{T}^3)}\|\psi\|_{L^2(\Omega)}.
\end{align*}
Here, for the last inequality we used the fact that $\psi$ and $w$ are periodic.
Consequently, we get
\begin{equation*}
\|S_1\psi\|_{L^2(\Omega,wdx)}\leq C\|w\|_{\mathcal{F}(\mathbb{T}^3)}\|\psi\|_{L^2(\Omega,w^{-1}dx)}
\end{equation*}
and \eqref{resol23} is now proved.

It remains to show the second estimate \eqref{spli}.
First we write
\begin{align}\label{wer}
&H_0(\delta_0+\rho)^{-1}\psi_j-
\sum_{\textbf{n}\in\mathbb{Z}^3}\frac{\widehat{\psi_j}(\textbf{n})e^{i\textbf{n}\cdot x}}
{|(\textbf{n}+\textbf{b})B|^2+z_j}\\
\nonumber&=\sum_{M=1}^\infty\sum_{\{\textbf{n}\in\mathbb{Z}^3:M-1\leq|\textbf{n}B|<M\}}
\frac{\widehat{\psi_j}(\textbf{n})e^{i\textbf{n}\cdot x}[|(\textbf{n}+\textbf{b})B|^2+z_j-(\textbf{n}+k)A(\textbf{n}+k)^T]}
{[(\textbf{n}+k)A(\textbf{n}+k)^T][|(\textbf{n}+\textbf{b})B|^2+z_j]}.
\end{align}
We then consider the second-order elliptic operator
$DAD^T$ on the torus $[0,2\pi)^3\approx\mathbb{R}^3/(2\pi\mathbb{Z})^3$
which has a complete set of eigenfunctions $\{e^{i\textbf{n}\cdot x}:\textbf{n}\in\mathbb{Z}^3\}$
with the corresponding eigenvalues $\{\textbf{n}A\textbf{n}^T:\textbf{n}\in\mathbb{Z}^3\}$.
Hence,
$$P_M\psi=\sum_{\{\textbf{n}\in\mathbb{Z}^3:\textbf{n}A\textbf{n}^T\in[(M-1)^2,M^2)\}}
\widehat{\psi}(\textbf{n})e^{i\textbf{n}\cdot x}
=\sum_{\{\textbf{n}\in\mathbb{Z}^3:|\textbf{n}B|\in[M-1,M)\}}
\widehat{\psi}(\textbf{n})e^{i\textbf{n}\cdot x}$$
is the projection of $\psi$ to the subspace of $L^2(\Omega)$,
spanned by eigenfunctions with eigenvalues in $[(M-1)^2,M^2)$.
In view of this fact, the following estimate for the spectral projection,
obtained by Sogge \cite{Sog} (see Theorem 2.2 (i) there),
can be used to bound the right-hand side of \eqref{wer} in $L^2$ space:
For $\psi\in L^2(\Omega)$ and $1\leq p\leq4/3$,
\begin{equation}\label{discre}
\|P_M\psi\|_{L^2}\leq CM^{\frac12(\frac6p-4)}\|\psi\|_{L^p}.
\end{equation}
See the proof of Lemma 5.2 in \cite{Sh}.
This spectral projection estimate can be thought of as a discrete version
of the Fourier restriction estimate of Stein and Tomas \cite{To}.
A weighted version of \eqref{discre},
\begin{equation}\label{df}
\|P_M\psi\|_{L^2(\mathbb{T}^3,wdx)}\leq CM^{1/2}\|w\|_{\mathcal{M}^p(\mathbb{T}^3)}\|\psi\|_{L^2(\mathbb{T}^3)},\quad1<p\leq3/2,
\end{equation}
can be also found in Section 5 of \cite{Sh3} where it was used
in handling \eqref{wer} in the weighted $L^2$ space with the weight $w$
in the Morrey space $\mathcal{M}^p(\mathbb{R}^3)$.
Note that \eqref{df} is the analog of the weighted $L^2$ Fourier restriction estimate obtained in \cite{CS,CR}.
In our case, we need the following lemma which
can be viewed as an extension of \eqref{df} to the class $\mathcal{F}$
because $\mathcal{M}^p(\mathbb{R}^3)\subset\mathcal{F}$.

\begin{prop}\label{sam}
Let $w\in L^1(\mathbb{T}^3)$.
Assume that $w\geq c_w>0$ and $\|w\|_{\mathcal{F}(\mathbb{T}^3)}<\infty$.
Then one has
\begin{equation}\label{proj}
\bigg\|\sum_{|\textbf{n}B|\in[k-1,k)}\widehat{\psi}(\textbf{n})
e^{i\textbf{n}\cdot x}\bigg\|_{L^2(\mathbb{T}^3,wdx)}\leq Ck^{1/2}\|w\|_{\mathcal{F}(\mathbb{T}^3)}^{1/2}\|\psi\|_{L^2(\mathbb{T}^3)}
\end{equation}
for $\psi\in L^2(\mathbb{T}^3)$ and $k\geq1$.
Here, $B=\sqrt{A}\geq0$ and $C$ depends only on $A$.
\end{prop}

Using \eqref{wer}, Minkowski's inequality, and this proposition which will be shown in the next section,
the left-hand side of \eqref{spli} is now bounded by
\begin{align}\label{endd}
C\|w\|_{\mathcal{F}(\mathbb{T}^3)}^{1/2}&\sum_{M=1}^\infty M^{1/2}\\
\nonumber\times&\bigg\|\sum_{|\textbf{n}B|\in[M-1,M)}
\frac{\widehat{\psi_j}(\textbf{n})e^{i\textbf{n}\cdot x}[|(\textbf{n}+\textbf{b})B|^2+z_j-(\textbf{n}+k)A(\textbf{n}+k)^T]}
{[(\textbf{n}+k)A(\textbf{n}+k)^T][|(\textbf{n}+\textbf{b})B|^2+z_j]}\bigg\|_{L^2(\mathbb{T}^3)}.
\end{align}
Then we recall from (5.10) in \cite{Sh} that
\begin{align*}
\sum_{M=1}^\infty M&\sup\bigg|\frac{|(\textbf{n}+\textbf{b})B|^2+z_j-(\textbf{n}+k)A(\textbf{n}+k)^T}
{[(\textbf{n}+k)A(\textbf{n}+k)^T][|(\textbf{n}+\textbf{b})B|^2+z_j]}\bigg|\\
&\leq
\sum_{M=1}^\infty M\sup\frac{|\rho|2^j}
{\big(\big|\,|(\textbf{n}+\textbf{b})B|^2-\rho^2a_1s_0\big|+2^j|\rho|\big)^2}\leq C.
\end{align*}
Here, the sup is taken over all $\textbf{n}\in\mathbb{Z}^3$
such that $n_1\in[2^{j-1},2^j-1]$ and $|\textbf{n}B|\in[M-1,M)$.
Using this and the dual estimate
\begin{equation*}
\bigg\|\sum_{|\textbf{n}B|\in[k-1,k)}\widehat{\psi}(\textbf{n})
e^{i\textbf{n}\cdot x}\bigg\|_{L^2(\mathbb{T}^3)}\leq Ck^{1/2}\|w\|_{\mathcal{F}(\mathbb{T}^3)}^{1/2}\|\psi\|_{L^2(\mathbb{T}^3,w^{-1}dx)}
\end{equation*}
of \eqref{proj}, \eqref{endd} is bounded as follows:
\begin{align*}
\eqref{endd}&\leq C\|w\|_{\mathcal{F}(\mathbb{T}^3)}^{1/2}\sum_{M=1}^\infty M^{1/2}
\bigg\|\sum_{|\textbf{n}B|\in[M-1,M)}
\widehat{\psi_j}(\textbf{n})e^{i\textbf{n}\cdot x}\bigg\|_{L^2(\mathbb{T}^3)}\\
&\qquad\qquad\qquad\qquad\quad\,
\times\sup\bigg|\frac{|(\textbf{n}+\textbf{b})B|^2+z_j-(\textbf{n}+k)A(\textbf{n}+k)^T}
{[(\textbf{n}+k)A(\textbf{n}+k)^T][|(\textbf{n}+\textbf{b})B|^2+z_j]}\bigg|\\
&\leq C\|w\|_{\mathcal{F}(\mathbb{T}^3)}
\|\psi_j\|_{L^2(\mathbb{T}^3,w^{-1}dx)}\\
&\qquad\qquad\qquad
\times\sum_{M=1}^\infty M\sup\bigg|\frac{|(\textbf{n}+\textbf{b})B|^2+z_j-(\textbf{n}+k)A(\textbf{n}+k)^T}
{[(\textbf{n}+k)A(\textbf{n}+k)^T][|(\textbf{n}+\textbf{b})B|^2+z_j]}\bigg|\\
&\leq C\|w\|_{\mathcal{F}(\mathbb{T}^3)}
\|\psi_j\|_{L^2(\mathbb{T}^3,w^{-1}dx)}.
\end{align*}
Hence we get \eqref{spli}.

\section{Proof of Proposition \ref{sam}}\label{sec5}

To show \eqref{proj} in Proposition \ref{sam}, we first let
$$\phi(x)=\sum_{|\textbf{n}B|[k-1,k)}\widehat{\psi}(\textbf{n})e^{i\textbf{n}\cdot x}$$
for $\psi\in L^2(\mathbb{T}^3)$.
Then it is clear that $\phi\in C^\infty(\mathbb{T}^3)$, and
\begin{equation}\label{sgh}
\|(DAD^T-k^2-2ik)\phi\|_{L^2(\mathbb{T}^3)}
\leq Ck\|\psi\|_{L^2(\mathbb{T}^3)}
\end{equation}
since $\phi$ is the projection of $\psi$ to the subspace of $L^2(\Omega)$,
spanned by eigenfunctions with eigenvalues in $[(k-1)^2,k^2)$.
Hence, if we show the following uniform Sobolev inequality
\begin{equation}\label{sdf}
\|\varphi\|_{L^2(\mathbb{T}^3,wdx)}\leq Ck^{-1/2}\|w\|_{\mathcal{F}(\mathbb{T}^3)}^{1/2}
\|(DAD^T-k^2-2ik)\varphi\|_{L^2(\mathbb{T}^3)}
\end{equation}
for $\varphi\in C^\infty(\mathbb{T}^3)$ and $k\geq1$,
then it follows from \eqref{sgh} that
\begin{align*}
\|\phi\|_{L^2(\mathbb{T}^3,wdx)}
&\leq Ck^{-1/2}\|w\|_{\mathcal{F}(\mathbb{T}^3)}^{1/2}\|(DAD^T-k^2-2ik)\phi\|_{L^2(\mathbb{T}^3)}\\
&\leq Ck^{1/2}\|w\|_{\mathcal{F}(\mathbb{T}^3)}^{1/2} \|\psi\|_{L^2(\mathbb{T}^3)}
\end{align*}
as desired.

Now we have to show \eqref{sdf} which can be thought of as an extension of that in \cite{Sh3}
for the Morrey class $\mathcal{M}^p(\mathbb{R}^3)$ to the class $\mathcal{F}$.
Fix $x_0\in\Omega$. Let $\widetilde{\eta}\in C_0^\infty(Q(x_0,1/2))$ be such that
$\widetilde{\eta}=1$ on $Q(x_0,1/4)$.
Here, $Q(x,r)$ denotes the cube centered at $x$ with side length $r$.
Then we note that
\begin{equation}\label{bes}
\varphi(x)\widetilde{\eta}(x)^2=
\widetilde{\eta}(x)\int_{\mathbb{R}^3}F_{z,B}(x-y)(DAD^T+z)(\varphi\widetilde{\eta})(y)dy,
\end{equation}
where $F_{z,B}(x)$ is the Fourier transform of $(|yB|^2+z)^{-1}$,
given by
\begin{align*}
F_{z,B}(x)&=\frac1{\textrm{det}(B)}\int_{\mathbb{R}^3}\frac{e^{-ixB^{-1}\cdot y}}{|y|^2+z}dy\\
&=\frac{c}{\textrm{det}(B)}\Big(\frac{z}{|xB^{-1}|}\Big)^{\frac12(\frac32-1)}
K_{\frac32-1}(\sqrt{z}|xB^{-1}|).
\end{align*}
Here, $K_{\frac32-1}$ denotes the modified Bessel function of the third kind of order $3/2-1$
(see \cite{L}, p. 108).
Since $x,y\in Q(x_0,1/2)$, $|x-y|<1$.
From this, we rewrite the right-hand side of \eqref{bes} as
\begin{equation*}
\widetilde{\eta}(x)\int_{\mathbb{R}^3}F_{z,B}(x-y)\eta(|x-y|)
[(DAD^T+z)\varphi\cdot\widetilde{\eta}-2D\varphi A(D\widetilde{\eta})^T-\varphi DAD^T\widetilde{\eta}](y)dy,
\end{equation*}
where $\eta\in C_0^\infty((-2.2))$ and $\eta(r)=1$ if $|r|\leq1$.
Now we assume for the moment that
\begin{equation}\label{dfg}
\bigg\|\int_{\mathbb{R}^3}F_{z,B}(x-y)\eta(|x-y|)f(y)dy\bigg\|_{L^2(\mathbb{R}^3,wdx)}
\leq\frac C{|z|^{1/4}}\|w\|_{\mathcal{F}(\mathbb{R}^3)}^{1/2}\|f\|_{L^2(\mathbb{R}^3)},
\end{equation}
where $\text{Re}\sqrt{z}\geq1$.
Using this, we then see that for $z=-(k+i)^2$
\begin{align}\label{ppo}
\nonumber&\|\varphi\widetilde{\eta}^2\|_{L^2(\mathbb{R}^3,wdx)}\\
\nonumber&\qquad\leq Ck^{-1/2}\|w\widetilde{\eta}^2\|_{\mathcal{F}(\mathbb{R}^3)}^{1/2}\\
\nonumber&\qquad\quad\times\big(\|(DAD^T+z)\varphi\cdot\widetilde{\eta}\|_{L^2(\mathbb{R}^3)}
+\|D\varphi A(D\widetilde{\eta})^T\|_{L^2(\mathbb{R}^3)}
+\|\varphi DAD^T\widetilde{\eta}\|_{L^2(\mathbb{R}^3)}\big)\\
\nonumber&\qquad\leq Ck^{-1/2}\|w\|_{\mathcal{F}(\mathbb{T}^3)}^{1/2}\\
&\qquad\quad\times\big(\|(DAD^T-k^2-2ik)\varphi\|_{L^2(\mathbb{T}^3)}
+\|D\varphi\|_{L^2(\mathbb{T}^3)}+\|\varphi\|_{L^2(\mathbb{T}^3)}\big),
\end{align}
where we used for the last inequality the fact that $w$ and $\psi$ are periodic.
We also see that for $k\geq1$
$$\|D\varphi\|_{L^2(\mathbb{T}^3)}+\|\varphi\|_{L^2(\mathbb{T}^3)}
\leq C\|(DAD^T-k^2-2ik)\varphi\|_{L^2(\mathbb{T}^3)}$$
using the Fourier series and Parseval's formula (\cite{SW}).
Combining this and \eqref{ppo}, we get
$$\|\varphi\|_{L^2(\mathbb{T}^3,wdx)}
\leq Ck^{-1/2}\|w\|_{\mathcal{F}(\mathbb{T}^3)}^{1/2}
\|(DAD^T-k^2-2ik)\varphi\|_{L^2(\mathbb{T}^3)}$$
as desired.

It remains to show \eqref{dfg}. But this follows from the following lemma
by using partition of unity and standard rescaling argument ({\it cf. \cite{Sh3}}, Theorem 5.33).
See also \cite{Sog} (pp. 134-135) for the case of $L^2\rightarrow L^q$ estimates.

\begin{lem}\label{lem5}
Let $w\geq0$ and $w\in\mathcal{F}(\mathbb{R}^3)$.
Assume that $a\in C^\infty(\mathbb{R}^3\times\mathbb{R}^3)$ and
\begin{equation}\label{supp}
\textrm{supp}\,a\subset\{(x,y)\in\mathbb{R}^3\times\mathbb{R}^3:1/2\leq|x-y|\leq2\}.
\end{equation}
Then, for $f\in L^2(\mathbb{R}^3)$,
\begin{equation}\label{osc}
\bigg\|\int_{\mathbb{R}^3}e^{i\lambda|x-y|}a(x,y)f(y)dy\bigg\|_{L^2(\mathbb{R}^3,wdx)}
\leq C|\lambda|^{-1/2}\|w\|_{\mathcal{F}(\mathbb{R}^3)}^{1/2}\|f\|_{L^2(\mathbb{R}^3)},
\end{equation}
where $\lambda\in\mathbb{R}$ with $|\lambda|\geq1$.
\end{lem}

\begin{rem}
This lemma is an extension of Theorem 5.5 in \cite{Sh3} for the Morrey class $\mathcal{M}^p(\mathbb{R}^3)$
to the class $\mathcal{F}$, which is also
a weighted version of an oscillatory integral theorem of Stein (see \cite{St}, p. 380).
\end{rem}

\begin{proof}[Proof of Lemma \ref{lem5}]
Using partition of unity and the assumption \eqref{supp}, we may first assume that
\begin{equation*}
\textrm{supp}\,a\subset\{(x,y)\in\mathbb{R}^3\times\mathbb{R}^3:
|x-x_0|<\delta,\, |y-y_0|<\delta,\, 1/2\leq|x-y|\leq2\}
\end{equation*}
for some $x_0,y_0\in\mathbb{R}^3$ and a sufficiently small $\delta>0$.
Since $1/|x|^2\in\mathcal{M}^p(\mathbb{R}^3)\subset\mathcal{F}$ clearly,
we may also assume that $w>0$ by replacing $w$ with
$\widetilde{w}(x)=w(x)+\varepsilon/|x|^2$ and then letting $\varepsilon\rightarrow0$.
By duality, \eqref{osc} is equivalent to
\begin{equation*}
\bigg\|\int_{\mathbb{R}^3}e^{i\lambda|x-y|}a(x,y)f(y)dy\bigg\|_{L^2(\mathbb{R}^3)}
\leq C|\lambda|^{-1/2}\|w\|_{\mathcal{F}(\mathbb{R}^3)}^{1/2}\|f\|_{L^2(\mathbb{R}^3,w^{-1}dx)}.
\end{equation*}
By translation we may assume here that $x_0=0$,
and since $a(x,y)=0$ if $|x_1|>\delta$, we only need to show that
\begin{equation}\label{osc3}
\int_{\mathbb{R}^2}\bigg|\int_{\mathbb{R}^3}e^{i\lambda|x-y|}a(x_1,x',y)f(y)dy\bigg|^2dx'
\leq C|\lambda|^{-1}\|w\|_{\mathcal{F}(\mathbb{R}^3)}\int_{\mathbb{R}^3}|f|^2w^{-1}dz
\end{equation}
for any fixed $x_1\in[-\delta,\delta]$. Here, $x'=(x_2,x_3)\in\mathbb{R}^2$.

To show \eqref{osc3}, we first let
\begin{equation*}
S_\lambda f(x')=\int_{\mathbb{R}^3}e^{i\lambda|x-y|}a(x_1,x',y)f(y)dy
\end{equation*}
for $x_1\in[-\delta,\delta]$ fixed.
Then, using the adjoint operator $S_\lambda^\ast$ of $S_\lambda$,
\eqref{osc3} follows easily from
\begin{equation}\label{osc4}
\|w^{1/2}S_\lambda^\ast S_\lambda(w^{1/2}f)\|_{L^2(\mathbb{R}^3)}
\leq C|\lambda|^{-1}\|w\|_{\mathcal{F}(\mathbb{R}^3)}\|f\|_{L^2(\mathbb{R}^3)},
\end{equation}
where
$$S_\lambda^\ast S_\lambda f(y)=\int_{\mathbb{R}^3}K_\lambda(y,z)f(z)dx$$
with
$$K_\lambda(y,z)=\int_{\mathbb{R}^2}e^{-i\lambda(|y-x|-|z-x|)}\overline{a}(x_1,x',y)a(x_1,x',y)dx'$$
such that
\begin{equation}\label{ker}
|K_\lambda(y,z)|\leq\frac{C}{1+|\lambda||y-z|}
\end{equation}
(see \cite{St}, p. 382).
The key point here is that the kernel $K_\lambda(y,z)$ can be controlled by that of the fractional integral operator $I_2$ of order $2$ given in \eqref{fi}.
Indeed, by \eqref{ker} it is clear that
$$|S_\lambda^\ast S_\lambda f(y)|\leq C|\lambda|^{-1}I_2(|f(z)|)(y).$$
Thus by the estimate \eqref{ee3} it follows that
\begin{align*}
\|S_\lambda^\ast S_\lambda f\|_{L^2(\mathbb{R}^3,wdy)}
&\leq C|\lambda|^{-1}\|I_2(|f(z)|)\|_{L^2(\mathbb{R}^3,wdy)}\\
&\leq C|\lambda|^{-1}\|w\|_{\mathcal{F}(\mathbb{R}^3)}\|f\|_{L^2(\mathbb{R}^3,w^{-1}dz)}
\end{align*}
which is clearly equivalent to \eqref{osc4}.
The proof is now completed.
\end{proof}


\end{document}